\documentclass[12pt]{amsart}
\usepackage{amsthm}
\usepackage{amssymb}
\usepackage{longtable}
\usepackage{geometry}
\usepackage{enumerate}
\usepackage{setspace}
\usepackage{pgfkeys}
\usepackage{upgreek}
\usepackage{tikz}
\usepackage{tikz-cd}
\usetikzlibrary{matrix, arrows}
\usepackage[unicode=true]
 {hyperref}
\hypersetup{
colorlinks=true,
urlcolor=black,
citecolor=blue,
linkcolor=blue,
}

\allowdisplaybreaks

\newtheorem{thm}{Theorem}[section]
\newtheorem{prop}[thm]{Proposition}
\newtheorem{lem}[thm]{Lemma}

\theoremstyle{definition}
\newtheorem{definition}[thm]{Definition}

\newtheorem{example}[thm]{Example}

\theoremstyle{remark}
\newtheorem{remark}[thm]{Remark}

\numberwithin{equation}{section}

\newcommand{\Perm}{\mathrm{Perm}}
\newcommand{\Hol}{\mathrm{Hol}}
\newcommand{\Aut}{\mathrm{Aut}}
\newcommand{\End}{\mathrm{End}}
\newcommand{\Map}{\mathrm{Map}}

\newcommand{\bN}{\mathbb{N}}
\newcommand{\E}{\mathcal{E}}
\newcommand{\ff}{\mathfrak{f}}
\newcommand{\fg}{\mathfrak{g}}

\newcommand{\Inn}{\mathrm{Inn}}
\newcommand{\Out}{\mathrm{Out}}
\newcommand{\Hom}{\mathrm{Hom}}

\newcommand{\fa}{\mathfrak{a}}
\newcommand{\fb}{\mathfrak{b}}

\newcommand{\fp}{\mathfrak{p}}
\newcommand{\inn}{\mathrm{\tiny inn}}
\newcommand{\out}{\mathrm{\tiny out}}

\newcommand{\proj}{\mathrm{proj}}

\begin{document}

\large 

\title[Hopf-Galois structures of isomorphic type]{Hopf-Galois structures of isomorphic type on a non-abelian characteristically simple extension}

\author{Cindy (Sin Yi) Tsang}
\address{School of Mathematics, Sun Yat-Sen University, Zhuhai}
\email{zengshy26@mail.sysu.edu.cn}\urladdr{http://sites.google.com/site/cindysinyitsang/} 

\date{\today}

\maketitle

\vspace{-2mm}

\begin{abstract}Let $L/K$ be a finite Galois extension whose Galois group $G$ is non-abelian and characteristically simple. Using tools from graph theory, we shall give a closed formula for the total number of Hopf-Galois structures on $L/K$ with associated group isomorphic to $G$.
\end{abstract}

\tableofcontents

\vspace{-4mm}

\section{Introduction}

Let $L/K$ be a finite Galois extension with Galois group $G$. Write $\Perm(G)$ for the symmetric group of $G$. Recall that a subgroup $\mathcal{N}$ of $\Perm(G)$ is said to be \emph{regular} if the map
\[ \xi_\mathcal{N}: \mathcal{N}\longrightarrow G;\hspace{1em}\xi_\mathcal{N}(\eta) = \eta(1)\]
is bijective, or equivalently, if the $\mathcal{N}$-action on $G$ is both transitive and free. For example, the images of the left and right regular representations
\[\begin{cases}
\lambda: G\longrightarrow \Perm(G);\hspace{1em}\lambda(\sigma) = (\tau\mapsto \sigma\tau),\\
\rho:G\longrightarrow\Perm(G);\hspace{1em}\rho(\sigma) = (\tau\mapsto\tau\sigma^{-1}),
\end{cases}\]
respectively, are plainly regular subgroups of $\Perm(G)$. By work of C. Greither and B. Pareigis \cite{GP}, each Hopf-Galois structure $\mathcal{H}$ on $L/K$ is associated to a regular subgroup $\mathcal{N}_\mathcal{H}$ of $\Perm(G)$ which is normalized by $\lambda(G)$, and the  \emph{type} of $\mathcal{H}$ is defined to be the isomorphism class of $\mathcal{N}_\mathcal{H}$. In particular, for any finite group $N$ of the same order as $G$, there is a one-to-one correspondence between Hopf-Galois structures on $L/K$ of type $N$ and elements in
\[ \E(G,N) = \left\{ \begin{array}{c}\mbox{regular subgroups of $\Perm(G)$ which are}
\\ \mbox{isomorphic to $N$ and normalized by $\lambda(G)$}\end{array}\right\}.\]
The enumeration of this set has since become an active line of research. For example, see work of L. N. Childs, N. P. Byott, and T. Kohl. One important result, which was proven by N. P. Byott in \cite{By96}, is the formula
\[\#\E(G,N) = \frac{|\Aut(G)|}{|\Aut(N)|}\cdot \#\left\{\begin{array}{c}\mbox{regular subgroups in $\Hol(N)$}\\\mbox{which are isomorphic to $G$}\end{array}\right\},\]
where $\Hol(N)$ denotes the \emph{holomorph} of $N$ and is given by
\begin{equation}\label{Hol(N)}\Hol(N)=\rho(N)\rtimes \Aut(N).\end{equation}
In particular, it suffices to study the set
\[ \E'(G,N) = \{\mbox{regular subgroups of $\Hol(N)$ isomorphic to $G$}\},\]
which is much easier to understand because of the nice description (\ref{Hol(N)}). See \cite[Chapter 2]{Childs book} for more background on the study of Hopf-Galois structures.

\vspace{1.5mm}

In this paper, we shall be interested in the Hopf-Galois structures on $L/K$ of type $G$, or equivalently, the regular subgroups lying in $\E'(G,G)$. Let
\[\proj_\Aut:\Hol(G)\longrightarrow \Aut(G)\]
denote the projection map given by (\ref{Hol(N)}), and write $\Inn(G)$ for the group of inner automorphisms on $G$. Define
\begin{align*}
\E'_\inn(G,G) & = \{\mathcal{N}\in\E'(G,G) :\proj_\Aut(\mathcal{N})\subset\Inn(G)\} ,\\
\E'_\out(G,G) & =\{\mathcal{N}\in \E'(G,G) : \proj_\Aut(\mathcal{N})\not\subset\Inn(G)\},
\end{align*}
and we shall consider them separately. Let us remark that $ \E'_\inn(G,G)$ always \par\noindent contains $\lambda(G)$ and $\rho(G)$, which coincide exactly when $G$ is abelian. Further, recall that a pair $(f,g)$ of endomorphisms on $G$ is said to be \emph{fixed point free} if $f(\sigma) = g(\sigma)$ holds precisely when $\sigma = 1$. Then, by work of N. P. Byott and L. N. Childs in \cite{Byott Childs}, such a pair gives rise to an element of $\E'_\inn(G,G)$, and 
\begin{equation}\label{inner equation} \#\E'_\inn(G,G) = \frac{1}{|\Aut(G)|}\cdot\#\left\{ \begin{array}{c}\mbox{fixed point free pairs $(f,g)$}\\ \mbox{of endomorphisms on $G$}\end{array}\right\}\end{equation}
when $G$ has trivial center; see \cite[Propositions 2 and 6]{Byott Childs}.

\vspace{1.5mm}

In the proof of \cite[Theorem 4]{Childs simple}, S. Carnahan and L. N. Childs showed that 
\[ \#\E'_\inn(G,G)=2\mbox{ and }\#\E'_\out(G,G) = 0
\]
when $G$ is non-abelian simple. Our main theorem is the following significant generalization of their result to the case when $G$ is non-abelian characteristically simple, that is, when $G$ is a direct product of copies of some non-abelian simple group. 

\begin{thm}\label{main thm}Suppose that $G$ is a direct product of $n\in\bN$ copies of a finite non-abelian simple group $T$. Then, we have
\[ \#\E'_\inn(G,G) = 2^n\cdot(n|\Aut(T)| + 1)^{n-1}\mbox{ and }\#\E'_\out(G,G) = 0.\]
\end{thm}

In the rest of this paper, we shall assume that
\[ G = T\times\cdots\times T\mbox{ ($n$ copies), where $T$ is any non-trivial finite group}.\]
Note that $T$ is not assumed to be non-abelian simple in general. Put
\[ \bN_n = \{1,\dots,n\}\mbox{ and }\bN_{0,n} = \{0,1,\dots,n\}.\]
For each $i\in\bN_n$, for brevity, define
\[ T^{(i)} = 1\times\cdots\times 1\times T\times 1
\times\cdots\times 1\mbox{ ($T$ is in the $i$th position)},\]
and write $x^{(i)}$ for an arbitrary element of $T^{(i)}$. For convenience, let us define $T^{(0)}$ to be the trivial subgroup, and write $x^{(0)}$ for the identity element. Now, let $\End^0(G)$ denote the group of all endomorphisms on $G$ of the shape
\[(x^{(1)},\dots,x^{(n)})\mapsto (\varphi_1(x^{(\theta(1))}),\dots,\varphi_n(x^{(\theta(n))})),\]
where $\theta \in \Map(\bN_n,\bN_{0,n})$, and for each $i\in\bN_n$, we have
\begin{equation}\label{varphi cond}\varphi_i\in\Hom(T^{(\theta(i))},T^{(i)})\mbox{ such that }\begin{cases}\varphi_i \mbox{ is trivial}&\mbox{if }\theta(i) = 0,\\\varphi_i\mbox{ is bijective}&\mbox{if }\theta(i)\neq0.
\end{cases}\end{equation}
Also, write $\Aut^0(G)$ for the subgroup consisting of those which are automorphisms, or equivalently
\begin{equation}\label{Aut0}\Aut^0(G) = \Aut(T)\mbox{ wr }S_n = \Aut(T)^n \rtimes S_n,\end{equation}
where $S_n$ denotes the symmetric group on $n$ letters. The wreath product ``wr'' here is the canonical one with $S_n$ acting naturally on $\bN_n$. The consideration of $\End^0(G)$ is motivated by the fact that
\begin{equation}\label{End eqn}
\End(G) = \End^0(G)\mbox{ and in particular }\Aut(G)= \Aut^0(G)
\end{equation}
when $T$ is non-abelian simple; see the proof of \cite[Lemma 3.2]{Byott simple}, for example. Using (\ref{End eqn}), as well as drawing tools from graph theory and group theory, respectively, we shall then prove the first and second equalities of Theorem~\ref{main thm}.


\section{Regular subgroups arising from inner automorphisms}\label{inn sec}

\subsection{Criteria for fixed point freeness}

Throughout this subsection, consider a pair $(f,g)$ with $f,g\in\End^0(G)$. Then, we have
\begin{align*}f(x^{(1)},\dots,x^{(n)}) &= (\varphi_{f,1}(x^{(\theta_f(1))}),\dots,\varphi_{f,n}(x^{(\theta_f(n))})),\\
g(x^{(1)},\dots,x^{(n)}) &= (\varphi_{g,1}(x^{(\theta_g(1))}),\dots,\varphi_{g,n}(x^{(\theta_g(n))})),\end{align*}
where $\theta_f,\theta_g\in\Map(\bN_n,\bN_{0,n})$, and $\varphi_{f,i},\varphi_{g,i}$ are as in (\ref{varphi cond}) for $i\in\bN_n$. Put
\[\uptheta_f = (\theta_f(1),\dots,\theta_f(n))\mbox{ and }\uptheta_g = (\theta_g(1),\dots,\theta_g(n)).\]
Using these $n$-tuples, we may associate to $(f,g)$ a graph as follows.

\begin{definition}For any two $n$-tuples $\upmu = (u_1,\dots,u_n)$ and $\upnu = (v_1,\dots,v_n)$ with entries in $\bN_{0,n}$, define $\Gamma_{\{\upmu,\upnu\}}$ to be the undirected multigraph with vertex set $\bN_{0,n}$, and for each $i\in\bN_n$, we draw one edge joining $u_i$ and $v_i$.
\end{definition}

\begin{definition}\label{graph def}Define $\Gamma_{\{f,g\}}$ to be the undirected multigraph associated to the $n$-tuples $\uptheta_f$ and $\uptheta_g$. Define $\Gamma_{(f,g)}$ to be the directed multigraph with vertex set $\bN_{0,n}$, and for each $i\in\bN_n$, we draw one arrow $\fa_i$ from $\theta_f(i)$ to $\theta_g(i)$ if $\varphi_{g,i}$ is bijective, as well as one arrow $\fb_i$ from $\theta_g(i)$ to $\theta_f(i)$ if $\varphi_{f,i}$ is bijective.
\end{definition}

By the condition in (\ref{varphi cond}), the multigraph $\Gamma_{(f,g)}$ may be obtained from $\Gamma_{\{f,g\}}$ via the following operations:
\begin{itemize}
\item Remove every loop at the vertex $0$.
\item Replace every edge $0\mbox{ --- }i$ by the arrow $0\longrightarrow i$ when $i\neq0$.
\item Replace every edge $i\mbox{ --- }j$ by the pair of arrows $i\longleftrightarrow j$ when $i,j\neq0$.
\end{itemize}
Note that in $\Gamma_{(f,g)}$ there is no arrow ending at the vertex $0$. Thus, a directed path in $\Gamma_{(f,g)}$ can start at the vertex $0$, but cannot pass through $0$ or end at $0$. We shall illustrate Definition~\ref{graph def} via the following example.

\begin{example}Take $n=4$. Suppose that
\begin{align*}f(x^{(1)},x^{(2)},x^{(3)},x^{(4)})& = (\varphi_{f,1}(x^{(0)}),\varphi_{f,2}(x^{(1)}),\varphi_{f,3}(x^{(2)}),\varphi_{f,4}(x^{(3)})),\\
g(x^{(1)},x^{(2)},x^{(3)},x^{(4)}) &= (\varphi_{g,1}(x^{(0)}),\varphi_{g,2}(x^{(0)}),\varphi_{g,3}(x^{(1)}),\varphi_{g,4}(x^{(3)})),
\end{align*}
where the $\varphi_{f,i},\varphi_{g,i}$ are as in (\ref{varphi cond}). Then, according to Definition~\ref{graph def}, we have
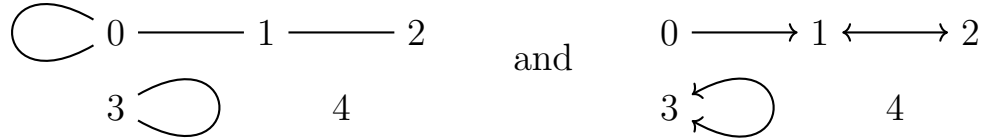
\[\begin{tikzpicture}[baseline = -0.5cm]
\node at (0,0) [name = 0] {$0$};
\node at (2,0) [name = 1] {$1$};
\node at (4,0) [name = 2] {$2$};
\node at (0,-1) [name = 3] {$3$};
\node at (3,-1) [name = 4] {$4$};
\draw[-, thick]
(0) -- (1) -- (2);
\draw[-, thick] (3) to [loop right, out=30, in=330, looseness = 12.5] (3);
\draw[-, thick] (0) to [loop left, out= 150, in=210, looseness = 12.5] (0);
\end{tikzpicture}\hspace{2em}\mbox{and}\hspace{2em}\begin{tikzpicture}[baseline=-0.5cm]
\node at (0,0) [name = 0] {$0$};
\node at (2,0) [name = 1] {$1$};
\node at (4,0) [name = 2] {$2$};
\node at (0,-1) [name = 3] {$3$};
\node at (3,-1) [name = 4] {$4$};
\path[<->,thick] (1) edge node {} (2);
\draw[<->, thick] (3) to [loop right, out=30, in=330, looseness = 12.5] (3);
\draw[->, thick]
(0) -- (1);
\end{tikzpicture}\]
for the graphs $\Gamma_{\{f,g\}}$ and $\Gamma_{(f,g)}$, respectively.
\end{example}

Let us briefly explain the ideas behind Definition~\ref{graph def}. To determine the solutions to $f(x) = g(x)$, we are reduced to considering, for each $i\in\bN_n$, the equation at the $i$th component given by
\[\varphi_{f,i}(x^{(\theta_f(i))}) = \varphi_{g,i}(x^{(\theta_g(i))}).\]
The edge $\mathfrak{e}_i$ joining $\theta_f(i)$ and $\theta_g(i)$ may be viewed as representing this equation, while the arrows $\fa_i$ and $\fb_i$ may be regarded as the homomorphisms
\[ \begin{cases}\gamma_{\fa_i} = \varphi_{g,i}^{-1}\circ\varphi_{f,i}&\mbox{if $\varphi_{g,i}$ is bijective},\\\gamma_{\fb_i} = \varphi_{f,i}^{-1}\circ\varphi_{g,i}&\mbox{if $\varphi_{f,i}$ is bijective},\end{cases}\]
respectively. Observe that $\gamma_{\fa_i}$ and $\gamma_{\fb_i}$ are inverses of each other if both $\varphi_{g,i}$ and $\varphi_{f,i}$ are bijective. Given a directed path $\fp$ in $\Gamma_{(f,g)}$, we may write it as a concatenation of arrows, say
\[ \fp = \mathfrak{c}_{i_m}\cdots \mathfrak{c}_{i_1},\mbox{ where }\mathfrak{c}_{i_k} \in \{\fa_{i_k},\fb_{i_k}\}\mbox{ for each }1\leq k\leq m,\]
and the concatenation is from right to left. Define 
\[\gamma_\fp = \gamma_{\mathfrak{c}_{i_m}}\circ\cdots\circ \gamma_{\mathfrak{c}_{i_1}}\]
in this case. Then, we clearly have the following lemma:

\begin{lem}\label{fg lemma}Let $\sigma = (\sigma^{(1)},\dots,\sigma^{(n)})\in G$. Then, we have $f(\sigma) = g(\sigma)$ if and only if $\sigma^{(h(\fp))} = \gamma_\fp(\sigma^{(t(\fp))})$ holds for all directed paths $\mathfrak{p}$ in $\Gamma_{(f,g)}$, where $h(\fp)$ and $t(\fp)$ denote its head and tail, respectively.
\end{lem}

Let us note that  Definition~\ref{graph def} and the forward implication of Lemma~\ref{fg lemma} are still valid even if $\varphi_{f,i},\varphi_{g,i}$ are only non-trivial but not necessarily bijective for $\theta_f(i),\theta_g(i)\neq0$. However, the analysis for determining when $(f,g)$ is fixed point free is much more complicated. For the purpose of this paper, we have thus restricted to the situation when the condition in (\ref{varphi cond}) holds.

\vspace{1.5mm}

We shall now give criteria for $(f,g)$ to be fixed point free in terms of properties of $\Gamma_{\{f,g\}}$. Let us point out that the condition in (\ref{varphi cond}) is crucial for some of the arguments to hold. In particular, it ensures that if we have a path in $\Gamma_{\{f,g\}}$ joining $i$ and $j$ which does not go through $0$ or end at $0$, then we have a directed path in $\Gamma_{(f,g)}$ from $i$ to $j$ as well.

\vspace{1.5mm}

Recall that a \emph{tree} is a connected graph which has no cycle. Equivalently, a \emph{tree} is a graph in which any two vertices can be connected by a unique simple path. For a graph $\Gamma$ with $m$ vertices, it is known that $\Gamma$ is a tree if and only if $\Gamma$ is connected and has exactly $m-1$ edges, for any $m\in\bN$.

\begin{prop}\label{fpf prop1}If $\Gamma_{\{f,g\}}$ is a tree, then $(f,g)$ is fixed point free.
\end{prop}
\begin{proof}Suppose that $\Gamma_{\{f,g\}}$ is a tree. By the connectedness of $\Gamma_{\{f,g\}}$ and (\ref{varphi cond}), for each $i\in\bN_n$, we have a directed path $\fp_i$ in $\Gamma_{(f,g)}$ from $0$ to $i$. Hence, whenever $f(\sigma) = g(\sigma)$, where $\sigma = (\sigma^{(1)},\dots,\sigma^{(n)})\in G$, we have $\sigma^{(i)} = \gamma_{\fp_i}(\sigma^{(0)}) = 1$ for all $i\in\bN_n$ by Lemma~\ref{fg lemma}. This shows that $(f,g)$ is fixed point free.
\end{proof}

Recall that an automorphism $\varphi$ on $T$ is said to be \emph{fixed point free} $\varphi(\sigma) = \sigma$ precisely when $\sigma = 1$. As the next example shows, the converse of Proposition~\ref{fpf prop1} is false in general, and the issue lies in the existence of fixed point free automorphisms on $T$.


\begin{example}Take $n=2$. Suppose that
\begin{align*} f(x^{(1)},x^{(2)}) &= (\varphi_{f,1}(x^{(1)}),\varphi_{f,2}(x^{(2)})),\\
g(x^{(1)},x^{(2)}) & = (\varphi_{g,1}(x^{(1)}),\varphi_{g,2}(x^{(2)})),\end{align*}
where the $\varphi_{f,i},\varphi_{g,2}$ are as in (\ref{varphi cond}). Then, according to Definition~\ref{graph def}, we have
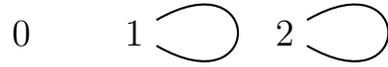
\vspace{-1mm}
\[\begin{tikzpicture}
\node at (0.5,0) [name = 0] {$0$};
\node at (2,0) [name = 1] {$1$};
\node at (4,0) [name = 2] {$2$};
\draw[-, thick] (1) to [loop right, out=30, in=330, looseness = 12.5] (1);
\draw[-, thick] (2) to [loop right, out=30, in=330, looseness = 12.5] (2);
\end{tikzpicture}\vspace{-1mm}\]
for the graph $\Gamma_{\{f,g\}}$, which is not a tree. But it is clear that the pair $(f,g)$ is fixed point free as long as both $\varphi_{f,1}^{-1}\circ\varphi_{g,1}$ and $\varphi_{f,2}^{-1}\circ\varphi_{g,2}$ are fixed point free.
\end{example}

Nevertheless, we have two partial converses of Proposition~\ref{fpf prop1}. Let us first make a crucial observation which gives us a way to construct fixed points of $(f,g)$. Suppose that $\Gamma_*$ is a connected component of $\Gamma_{\{f,g\}}$ not containing $0$ and fix some vertex $i_0$ in $\Gamma_*$. For each vertex $i$ in $\Gamma_*$, by connectedness and (\ref{varphi cond}), we have a directed path $\fp_i$ in $\Gamma_*$ from $i_0$ to $i$. In the case that $\Gamma_*$ has no cycle, there is essentially a unique choice of $\fp_i$, except that it could have consecutive repeated paths going in opposite directions, and hence $\gamma_{\fp_i}$ does not depend upon the choice of $\fp_i$. In the case that $\Gamma_*$ has exactly one simple cycle, suppose that it goes through $i_0$, and let $\mathfrak{q},\mathfrak{q}^{-1}$ denote the corresponding directed simple cycles based at $i_0$. Then, the homomorphism $\gamma_{\fp_i}$ depends on the direction and the number of times $\fp_i$ goes through the simple cycle. But $\gamma_{\fp_i}(\sigma^{(i_0)})$ does not depend upon the choice of $\fp_i$, when $\sigma^{(i_0)}\in T^{(i_0)}$ is a fixed point of the automorphisms $\gamma_{\mathfrak{q}}$ and $\gamma_{\mathfrak{q}^{-1}}$. Note that $\gamma_{\mathfrak{q}}$ and $\gamma_{\mathfrak{q}^{-1}}$ have the same fixed points. We shall illustrate the above discussion via the next example. Let us remark that the case that $\Gamma_*$ has two or more simple cycles need not be considered, as the proof of Proposition~\ref{fpf prop2} below shows.


\begin{example}\label{eg}Take $n=4$.
\begin{enumerate}[(a)]
\item Suppose that
\begin{align*}\hspace{5mm}f(x^{(1)},x^{(2)},x^{(3)},x^{(4)})& = (\varphi_{f,1}(x^{(0)}),\varphi_{f,2}(x^{(1)}),\varphi_{f,3}(x^{(3)}),\varphi_{f,4}(x^{(4)})),\\
g(x^{(1)},x^{(2)},x^{(3)},x^{(4)}) &= (\varphi_{g,1}(x^{(0)}),\varphi_{g,2}(x^{(2)}),\varphi_{g,3}(x^{(2)}),\varphi_{g,4}(x^{(2)})),
\end{align*}
where the $\varphi_{f,i},\varphi_{g,i}$ are as in (\ref{varphi cond}). According to Definition~\ref{graph def}, we have
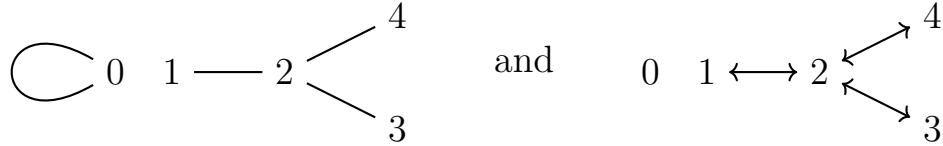
\[\hspace{5mm}\begin{tikzpicture}[baseline = 0cm]
\node at (0.75,0) [name = 0] {$0$};
\node at (1.5,0) [name = 1] {$1$};
\node at (3,0) [name = 2] {$2$};
\node at (4.5,-0.75) [name = 3] {$3$};
\node at (4.5,0.75) [name = 4] {$4$};
\draw[-, thick]
(1) -- (2) (2) -- (3) (2) -- (4);
\draw[-, thick] (0) to [loop left, out= 150, in=210, looseness = 12.5] (0);
\end{tikzpicture}\hspace{2em}\mbox{and}\hspace{2em}\begin{tikzpicture}[baseline=0cm]
\node at (0.75,0) [name = 0] {$0$};
\node at (1.5,0) [name = 1] {$1$};
\node at (3,0) [name = 2] {$2$};
\node at (4.5,-0.75) [name = 3] {$3$};
\node at (4.5,0.75) [name = 4] {$4$};
\draw[<->,thick] (1) -- (2);
\draw[<->,thick] (2) -- (3);
\draw[<->,thick] (2) -- (4);
\end{tikzpicture}\]
for the graphs $\Gamma_{\{f,g\}}$ and $\Gamma_{(f,g)}$, respectively. Observe that
\[\hspace{5mm} \fb_3\fa_2 = (1 \rightarrow 2 \rightarrow 3)\]
is the unique simple directed path going from $1$ to $3$. Since the connected component of $1$ has no cycle and does not contain the vertex $0$, any other directed path $\fp$ from $1$ to $3$ may be written as
\[\hspace{5mm}\fp = (\mathfrak{r}_3^{-1}\mathfrak{r}_3)\fb_3(\mathfrak{r}_2^{-1}\mathfrak{r}_2)\fa_2(\mathfrak{r}_1^{-1}\mathfrak{r}_1).\]
Here, for each $i=1,2,3$, the symbol $\mathfrak{r}_i$ denotes a possibly empty directed path beginning at vertex $i$, and $\mathfrak{r}_i^{-1}$ is the same path going in the opposite direction. Since $\gamma_{\mathfrak{r}_i^{-1}}$ and $\gamma_{\mathfrak{r}_i}$ are inverses of each other, it follows that $\gamma_{\fp}$ is equal to $\gamma_{\fb_3\fa_2}$, and hence is independent of $\fp$. 
\vspace{1.5mm}
\item Suppose that
\begin{align*}\hspace{5mm}f(x^{(1)},x^{(2)},x^{(3)},x^{(4)})& = (\varphi_{f,1}(x^{(1)}),\varphi_{f,2}(x^{(3)}),\varphi_{f,3}(x^{(3)}),\varphi_{f,4}(x^{(3)})),\\
g(x^{(1)},x^{(2)},x^{(3)},x^{(4)}) &= (\varphi_{g,1}(x^{(2)}),\varphi_{g,2}(x^{(1)}),\varphi_{g,3}(x^{(2)}),\varphi_{g,4}(x^{(4)})),
\end{align*}
where the $\varphi_{f,i},\varphi_{g,i}$ are as in (\ref{varphi cond}). According to Definition~\ref{graph def}, we have
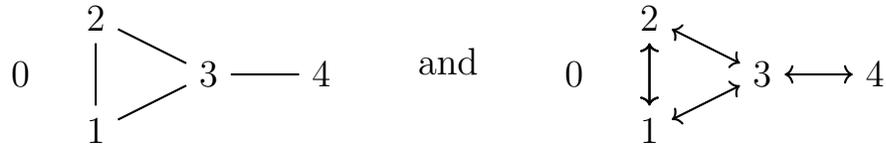
\[\hspace{5mm}\begin{tikzpicture}[baseline = 0cm]
\node at (0.5,0) [name = 0] {$0$};
\node at (1.5,-0.75) [name = 1] {$1$};
\node at (1.5,0.75) [name = 2] {$2$};
\node at (3,0) [name = 3] {$3$};
\node at (4.5,0) [name = 4] {$4$};
\draw[-, thick]
(1) -- (2) (1) -- (3) (2) -- (3) (3) -- (4);
\end{tikzpicture}\hspace{2em}\mbox{and}\hspace{2em}\begin{tikzpicture}[baseline=0cm]
\node at (0.5,0) [name = 0] {$0$};
\node at (1.5,-0.75) [name = 1] {$1$};
\node at (1.5,0.75) [name = 2] {$2$};
\node at (3,0) [name = 3] {$3$};
\node at (4.5,0) [name = 4] {$4$};
\path[<->,thick] (1) edge node {} (2);
\draw[<->, thick] (1) -- (2);
\draw[<->, thick] (1) -- (3);
\draw[<->, thick] (2) -- (3);
\draw[<->, thick] (3) -- (4);
\end{tikzpicture}\]
for the graphs $\Gamma_{\{f,g\}}$ and $\Gamma_{(f,g)}$, respectively. Define
\begin{align*}
\hspace{5mm} \mathfrak{q} & =  \fa_2\fb_3\fa_1 = (1\rightarrow2\rightarrow3\rightarrow1),\\
\mathfrak{q}^{-1} & = \fb_1\fa_3\fb_2 = (1\rightarrow3\rightarrow2\rightarrow1),
\end{align*}
which are the two directed simple cycles based at $1$. Since the connected component of $1$ has only one simple cycle and does not contain the vertex $0$, a directed path $\fp$ from $1$ to $4$ without consecutive repeated edges going in opposite directions may be written as
\[ \hspace{5mm}\fp = \fa_4\fb_3\fa_1\mathfrak{q}^m = (1\rightarrow2\rightarrow3\rightarrow4)\mathfrak{q}^m \mbox{ or }\fp = \fa_4\fb_2\mathfrak{q}^{-m} = (1\rightarrow3\rightarrow4)\mathfrak{q}^{-m}\]
for some non-negative integer $m$. The homomorphism $\gamma_\fp$ does depend on the choice of $\fp$. But as long as $\sigma^{(1)}\in T^{(1)}$ is a fixed point of $\gamma_{\mathfrak{q}}$, we have
\[\hspace{5mm}  \gamma_{\fa_2\fb_3\fa_1}(\sigma^{(1)}) = \sigma^{(1)}\mbox{ and so } \gamma_{\fb_3\fa_1}(\sigma^{(1)}) = \gamma_{\fb_2}(\sigma^{(1)}).\]
For any $m\in\mathbb{Z}$, the element $\sigma^{(1)}$ is also a fixed point of the automorphism $\gamma_{\mathfrak{q}^m}$, and we see that
\[\hspace{5mm}
\gamma_{\fa_4\fb_3\fa_1\mathfrak{q}^m}(\sigma^{(1)})=(\gamma_{\fa_4}\circ\gamma_{\fb_3\fa_1})(\sigma^{(1)})= (\gamma_{\fa_4}\circ\gamma_{\fb_2})(\sigma^{(1)}) = \gamma_{\fa_4\fb_2\mathfrak{q}^m}(\sigma^{(1)}).\]
It follows that the element $\gamma_{\fp}(\sigma^{(1)})$ is independent of $\fp$.
\end{enumerate}
\end{example}

\begin{prop}\label{fpf prop2}If $(f,g)$ is fixed point free, then the connected component of $\Gamma_{\{f,g\}}$ containing $0$ is a tree, and in each of the other connected components of $\Gamma_{\{f,g\}}$, the number of edges coincides with the number of vertices.
\end{prop}
\begin{proof}Suppose that $(f,g)$ is fixed point free. Let $\Gamma_0,\Gamma_1,\dots,\Gamma_r$, with $r\geq 0$, denote the connected components of $\Gamma_{\{f,g\}}$, such that $0$ lies in $\Gamma_0$. For each $0\leq k\leq r$, write $v_k$ and $e_k$, respectively, for the number of vertices and edges in $\Gamma_k$, as well as note that $e_k\geq v_k - 1$ because $\Gamma_k$ is connected. Also, we have
\[ n + 1 = v_0 + v_1 + \cdots + v_r\mbox{ and }n = e_0 + e_1 + \cdots + e_r \]
by definition. Below, we shall show that $e_k\geq v_k$ for all $1\leq k\leq r$. Together with the above equalities, this implies that $e_0\leq v_0-1$. We then deduce that in fact $e_0 = v_0 -1$, namely $\Gamma_0$ is a tree, and that $e_k = v_k$ for $1\leq k\leq r$.

\vspace{1.5mm}

Suppose for contradiction that $e_{k_0} = v_{k_0}-1$, namely $\Gamma_{k_0}$ is a tree, for some $1\leq k_0\leq r$. Let $i_{0}$ be any vertex in $\Gamma_{k_0}$ and fix some non-trivial element $\sigma^{(i_0)}$ in $T^{(i_{0})}$. For any vertex $i\neq i_{0}$ in $\Gamma_{k_0}$, by connectedness and (\ref{varphi cond}), there is a directed path $\fp_i$ in $\Gamma_{k_0}$ from $i_0$ to $i$. Notice that $\gamma_{\fp_i}$ does not depend on the choice of $\fp_i$ by the discussion prior to Example~\ref{eg}. For $i\neq i_0$, define
\[ \sigma^{(i)} = \begin{cases}\gamma_{\fp_i}(\sigma^{(i_0)}) &\mbox{if $i$ is in $\Gamma_{k_0}$},\\ 1 & \mbox{if $i$ is not in $\Gamma_{k_0}$},\end{cases}\]
and put $\sigma = (\sigma^{(1)},\dots,\sigma^{(n)})$ But then $f(\sigma) = g(\sigma)$ by Lemma~\ref{fg lemma} and $\sigma\neq1$. This contradicts that $(f,g)$ is fixed point free.
\end{proof}

\begin{prop}\label{fpf prop3}Suppose that $T$ does not admit any fixed point free automorphism. If $(f,g)$ is fixed point free, then $\Gamma_{\{f,g\}}$ is a tree.
\end{prop}
\begin{proof}Suppose that $(f,g)$ is fixed point free. Suppose also for contradiction that  $\Gamma_{\{f,g\}}$ is not a tree, namely it is not connected, and let $\Gamma_*$ be a connected component not containing $0$. By Proposition~\ref{fpf prop2}, we have exactly one simple cycle in $\Gamma_*$. Then, by (\ref{varphi cond}), we have a corresponding directed simple cycle $\mathfrak{q}$ in $\Gamma_{(f,g)}$, based at the vertex $i_0$ say. Note that $\gamma_{\mathfrak{q}}$ is an automorphism on $T^{(i_0)}$, which cannot be fixed point free by hypothesis, and hence $\gamma_{\mathfrak{q}}(\sigma^{(i_0)}) = \sigma^{(i_0)}$ for some non-trivial element $\sigma^{(i_0)}$ in $T^{(i_0)}$.

\vspace{1.5mm}

For each vertex $i\neq i_{0}$ in $\Gamma_{*}$, by connectedness and (\ref{varphi cond}), there is a directed path $\fp_i$ in $\Gamma_{*}$ from $i_0$ to $i$. Notice that $\gamma_{\fp_i}(\sigma^{(i_0)})$ does not depend on the choice of $\fp_i$ by the discussion prior to Example~\ref{eg}. For $i\neq i_0$, define
\[ \sigma^{(i)} = \begin{cases}\gamma_{\fp_i}(\sigma^{(i_0)}) &\mbox{if $i$ is in $\Gamma_{*}$},\\ 1 & \mbox{if $i$ is not in $\Gamma_{*}$},\end{cases}\]
and put $\sigma = (\sigma^{(1)},\dots,\sigma^{(n)})$. But then $f(\sigma) = g(\sigma)$ by Lemma~\ref{fg lemma} and $\sigma\neq1$. This contradicts that $(f,g)$ is fixed point free.
\end{proof}

\begin{remark}\label{remark}By the classification theorem of finite simple groups, any finite insolvable group has no fixed point free automorphism; see \cite[Theorem 1.48]{G book}.\end{remark}

\subsection{Proof of Theorem~\ref{main thm}: first statement} Put
\[ \mathcal{F}(G,G) = \{(f,g)\in\End^0(G)\times\End^0(G): \Gamma_{\{f,g\}}\mbox{ is a tree}\}.\]
First, we shall prove the following general statement:

\begin{prop}\label{main prop}We have
\[ \#\mathcal{F}(G,G) = 2^n\cdot n!\cdot |\Aut(T)|^n\cdot \left(n|\Aut(T)|+1\right)^{n-1}.\]
\end{prop}
\begin{proof}Observe that for any tree $\Gamma$ with vertex set $\bN_{0,n}$, which by definition has exactly $n$ edges, we have the equality
\[ \#\{(\upmu,\upnu)\in(\bN_{0,n})^n\times(\bN_{0,n})^n : \Gamma_{\{\upmu,\upnu\}} = \Gamma\} = 2^n\cdot n!.\]
This is because we have $2^n\cdot n!$ ways to pick an orientation for each edge and then label the $n$ arrows as $\mathfrak{e}_i$ for $i\in\bN_{n}$. Once such a choice is made, define the $i$th entries of $\upmu$ and $\upnu$, respectively, to be the tail and the head of $\mathfrak{e}_i$. We then have $\Gamma_{\{\upmu,\upnu\}}=\Gamma$, and the fact that $\Gamma$ has no cycle implies that different choices give rise to different pairs $(\upmu,\upnu)$.

\vspace{1.5mm}

Now, for any $\upmu,\upnu\in(\bN_{0,n})^n$, say $\upmu = (u_1,\dots,u_n)$ and $\upnu = (v_1,\dots,v_n)$, put
\[ d(\upmu,\upnu) = \#\{i\in\bN_{0,n}:u_i=0\}+\#\{i\in\bN_{0,n}:v_i=0\},\]
which is also equal to the degree of the vertex $0$ in $\Gamma_{\{\upmu,\upnu\}}$. Then, we have
\[ \#\{(f,g)\in\End^0(G)\times \End^0(G) : (\uptheta_f,\uptheta_g) = (\upmu,\upnu)\} = |\Aut(T)|^{2n-d(\upmu,\upnu)}\]
by (\ref{varphi cond}), and note that for $\Gamma_{\{\upmu,\upnu\}}$ to be a tree, necessarily $1\leq d(\upmu,\upnu)\leq n$.

\vspace{1.5mm}

For each integer $1\leq d\leq n$, let $\mathcal{T}_{n}(d)$ be the number of labelled trees on $n+1$ vertices, labelled by elements of $\bN_{0,n}$, in which the vertex $0$ has degree $d$, where two such labelled trees are regarded as distinct if and only if there is a pair of vertices which are joined by an edge in one tree but not in the other. Then, from the above discussion, it follows that
\[ \#\mathcal{F}(G,G) = \sum_{d=1}^{n}\left(\mathcal{T}_n(d)\cdot 2^n\cdot n!\cdot |\Aut(T)|^{2n-d}\right).\]
For each $1\leq d\leq n$, it was shown in \cite{Clarke} that
\[ \mathcal{T}_n(d) = {n-1 \choose d-1}n^{n-d}.\]
A simple calculation using the binomial theorem then yields the claim.
\end{proof}

Now, suppose that $T$ is non-abelian simple. Then, for any $f,g\in\End^0(G)$, 

\noindent by Propositions~\ref{fpf prop1} and~\ref{fpf prop3} as well as Remark~\ref{remark}. we have
\[\mbox{the pair $(f,g)$ is fixed point free if and only if $\Gamma_{\{f,g\}}$ is a tree}.\]
From (\ref{inner equation}) and (\ref{End eqn}), we see that
\[ \#\E'_\inn(G,G) = \frac{1}{|\Aut(T)|^n\cdot n!}\cdot\#\mathcal{F}(G,G).\]
The first statement in Theorem~\ref{main thm} now follows from Proposition~\ref{main prop}.

\section{Regular subgroups arising from outer automorphisms}\label{out sec}

\subsection{Criteria for regularity} Throughout this subsection, consider a subgroup $\mathcal{N}$ of $\Hol(G)$ isomorphic to $G$, with $\proj_{\Aut}(\mathcal{N})\subset\Aut^0(G)$.  As noted in \cite[Proposition 2.1]{Tsang HG}, which follows from (\ref{Hol(N)}), we have
\begin{equation}\label{N shape} \mathcal{N} = \{\rho(\fg(\sigma))\cdot \ff(\sigma):\sigma\in G\},\mbox{ where }\begin{cases}\ff\in\Hom(G,\Aut^0(G)) \\\fg\in\Map(G,G) \end{cases}\end{equation}
are such that
\[\label{g prop}\fg(\sigma\tau) = \fg(\sigma)\cdot \ff(\sigma)(\fg(\tau))\mbox{ for all }\sigma,\tau\in G.\]
Moreover, as one easily sees, we have
\begin{equation}\label{regularity}\mathcal{N}\mbox{ is regular if and only if }\fg\mbox{ is bijective}.
\end{equation}
Recall (\ref{Aut0}) and then define $\ff_{S_n}$ to be the homomorphism $\ff$ composed with the natural projection map $\Aut^0(G)\longrightarrow S_n$. Let us first make the following observation, which is motivated by an argument in \cite[p. 84]{Childs simple}.

\begin{prop}\label{out prop1}Suppose that the outer automorphism group $\Out(T)$ of $T$ is solvable and that $\ker(\ff_{S_n})$ is perfect. Then, we have $\ff(\ker(\ff_{S_n}))\subset\Inn(G)$. 
\end{prop}
\begin{proof}Observe that the homomorphism
\[\begin{tikzcd}[column sep = 1.5cm] \ker(\ff_{S_n})\arrow{r}{\ff}& \Aut(T)^n \arrow{r}{\mbox{\tiny quotient}} &\Out(T)^n,\end{tikzcd}\]
is trivial because $\ker(\ff_{S_n})$ is perfect but $\Out(T)$ is solvable. Thus, indeed the image of $\ker(\ff_{S_n})$ under $\ff$ lies in $\Inn(T)^n$, which is equal to $\Inn(G)$.\end{proof}

\begin{remark}\label{remark'}By Schreier's conjecture, which is a consequence of the classification theorem of finite simple groups, the outer automorphism group of any finite non-abelian simple group is solvable; see \cite[Theorem 1.46]{G book}.
\end{remark}

Next, we shall  investigate when it is possible  for $\mathcal{N}$ to be regular, or equivalently, for $\fg$ to be bijective.

\vspace{1.5mm}

Given any $\sigma\in G$, the assumption that $\ff(G)$ lies in $\Aut^0(G)$ implies 
\begin{equation}\label{fsigma} \ff(\sigma)(x^{(1)},\dots,x^{(n)}) = (\varphi_{\sigma,1}(x^{(\theta_\sigma(1))}),\dots,\varphi_{\sigma,n}(x^{(\theta_{\sigma}(n))})),\end{equation}
where $\theta_\sigma = \ff_{S_n}(\sigma)$, and $\varphi_{\sigma,i}\in\Aut(T)$ sends $T^{(\theta_{\sigma}(i))}$ to $T^{(i)}$. Also, write
\[ \fg(\sigma) = (a_{\sigma}^{(1)},\dots,a_{\sigma}^{(n)}).\]
In the above notation, we then have the following lemma:

\begin{lem}\label{relations lem} Let $\sigma,\tau\in G$ be such that $\sigma\tau = \tau\sigma$ and $\tau\in\ker(\ff_{S_n})$. Then, for all $i\in\bN_n$, we have the relation
\[\varphi_{\sigma,i}(a_\tau^{(\theta_\sigma(i))}) = (a_\sigma^{(i)})^{-1}\cdot a_\tau^{(i)}\cdot\varphi_{\tau,i}(a_\sigma^{(i)}).\]
\end{lem}
\begin{proof}The hypothesis $\tau\in\ker(\ff_{S_n})$ implies that
\begin{align*}
\fg(\sigma\tau) & = \fg(\sigma)\cdot \ff(\sigma)(\fg(\tau)) =  (a_\sigma^{(1)}\varphi_{\sigma,1}(a_\tau^{(\theta_\sigma(1))}),\dots,a_\sigma^{(n)}\varphi_{\sigma,n}(a_\tau^{(\theta_\sigma(n))})),\\
\fg(\tau\sigma) & = \fg(\tau)\cdot\ff(\tau)(\fg(\sigma)) =  (a_\tau^{(1)}\varphi_{\tau,1}(a_\sigma^{(1)}),\dots, a_\tau^{(n)}\varphi_{\tau,n}(a_\sigma^{(n)})).
\end{align*}
Since $\sigma\tau = \tau\sigma$, the claim is now clear.
\end{proof}

In what follows, fix a prime $p$ dividing $|T|$. For each $i\in\bN_n$, further choose a subgroup $P^{(i)}$ of $T^{(i)}$ of order $p$, and put
\begin{equation}\label{H def} H = P^{(1)}\times\cdots\times P^{(n)},\end{equation}
which is an elementary abelian $p$-group of rank $n$. Write $m\geq0$ for the rank of $\ff_{S_n}(H)$, and let us consider the $\ff_{S_n}(H)$-action on $\bN_n$; the restriction to the subgroup $H$ is only for convenience. We may decompose
\[ \bN_n = X_0\sqcup X_1\sqcup\cdots\sqcup X_r, \mbox{ with }r\geq 0,\]
where $X_0$ is the set of fixed points and $X_1,\dots,X_r$ are the non-trivial orbits 

\noindent of $\bN_n$ under the $\ff_{S_n}(H)$-action. For each $1\leq k\leq r$, let us fix a representative $i_k\in X_k$, and for each $i\in X_k$, choose an element $\sigma_i\in H$ such that $\theta_{\sigma_i}(i_k) = i$. 

\begin{lem}\label{g bound lem}Suppose that for all $i\in\bN_n\setminus X_0$, the element $\sigma_i\in H$ commutes with every element of $\ker(\ff_{S_n})$. Then, the image of $\ker(\ff_{S_n})$ under $\fg$ lies in
\[\prod_{i\in X_0} T^{(i)}\times\prod_{k=1}^{r}\left\{\prod_{i\in X_k}\varphi_{\sigma_i,i_k}^{-1}\left((a_{\sigma_i}^{(i_k)})^{-1}\cdot a\cdot\varphi(a_{\sigma_i}^{(i_k)})\right):a\in T^{(i_k)},\varphi\in\Aut(T^{(i_k)})\right\}.\]
Moreover, in the case that $\ff(\ker(\ff_{S_n}))\subset\Inn(G)$, for each $1\leq k\leq r$, we may replace $\Aut(T^{(i_k)})$ by $\Inn(T^{(i_k)})$ in the above, and in particular, we have
\begin{equation}\label{inequality}\#\fg(\ker(\ff_{S_n}))\leq |T|^{\#X_0}\cdot(|T||\Inn(T)|)^{r}\leq |T|^{\#X_0 + 2r}.\end{equation}
\end{lem}
\begin{proof}This follows immediately from Lemma~\ref{relations lem}.
\end{proof}
%
%

\begin{lem}\label{group lem}For each $1\leq k\leq r$, we have $\#X_k = p^{m_k}$ for some $m_k\in\bN$. In addition, we have the relations
\[ n - \#X_0 = \sum_{k=1}^{r} p^{m_k}\mbox{ and } m \leq \sum_{k=1}^{r} m_k.\]
\end{lem}
\begin{proof}The first claim is clear because $X_1,\dots,X_r$ are non-trivial orbits under the action of a $p$-group. The equality is also obvious from the definition.

\vspace{1.5mm}
 
To prove the inequality, given any subset $X$ of $\bN_n$, denote by $S_X$ its symmetric group regarded as a subgroup of $S_n$. For each $1\leq k\leq r$, define $\Delta_k$ to be the image, which is an elementary abelian group, of the homomorphism
\[\begin{tikzcd}[column sep = 2cm]\ff_{S_n}(H)\arrow{r}{\mbox{\tiny inclusion}}& S_n \arrow{r}{\tiny \theta\mapsto \theta|_{X_k}} & S_{X_k}\arrow{r}{\mbox{\tiny inclusion}} &S_n.\end{tikzcd}\]
Plainly, the $\Delta_k$-action on $X_k$ is transitive. For any $\theta\in\ff_{S_n}(H)$, since $\ff_{S_n}(H)$ is abelian, if $\theta$ fixes an element of the orbit $X_k$, then $\theta$ fixes all elements of $X_k$. This means the $\Delta_k$-action on $X_k$ is also free. Thus, we have $|\Delta_k| = \#X_k$, and so $\Delta_k$ has rank equal to $m_k$. Notice that $\ff_{S_n}(H)$ is contained in $\Delta_1\times\cdots\times\Delta_r$. Since $m$ is defined to be the $p$-rank of $\ff_{S_n}(H)$, the stated equality follows.
\end{proof}

\begin{prop}\label{out prop2}Suppose that $\ff_{S_n}(H)\neq1$ and that $|\ff_{S_n}(G)| = |T|^m$. If $\fg$ is bijective and (\ref{inequality}) holds, then necessarily $p\leq 3$.
\end{prop}
\begin{proof}The hypothesis implies that $r\geq1$ and that $|\ker(\ff_{S_n})| = |T|^{n-m}$. Suppose now that $\fg$ is bijective and that (\ref{inequality}) holds. Then, we have
\[ |T|^{n-m} \leq |T|^{\# X_0 + 2r}\mbox{ and so }n-\#X_0 \leq 2r+m.\]
Using Lemma~\ref{group lem}, we further deduce that
\[ \sum_{k=1}^{r}p^{m_k} \leq 2r + \sum_{k=1}^{r}m_k\mbox{ and so }\sum_{k=1}^{r}(p^{m_k}-2)\leq \sum_{k=1}^{r}m_k.\]
For $p\geq 5$, we have $p^x - 2 > x$ for all $x\geq1$. Hence, we must have $p\leq3$ for the above inequality to hold.
\end{proof}

\subsection{Proof of Theorem~\ref{main thm}: second statement} Suppose that $T$ is non-abelian simple. By (\ref{End eqn}), a subgroup $\mathcal{N}$ of $\Hol(G)$ isomorphic to $G$ is of the shape (\ref{N shape}), and we may use the same notation as in the previous subsection. 

\vspace{1.5mm}

It is known and is not hard to show that the normal subgroups of $G$ are exactly all the products among $T^{(1)},\dots,T^{(n)}$. Thus, there exists $0\leq m\leq n$, and also distinct $i_1,\dots,i_{n-m}\in\bN_n$, such that
\[\label{ker} \ker(\ff_{S_n}) = T^{(i_1)}\times\cdots\times T^{(i_{n-m})}.\]
Put $\mathbb{I} = \bN_n\setminus\{i_1,\dots,i_{n-m}\}$, which has size $m$. The above implies that
\[ \ff_{S_n}(G) \simeq \frac{G}{\ker(\ff_{S_n})}\simeq \prod_{i\in \mathbb{I}} T^{(i)},\]
and so it has order equal to $|T|^m$. Let $H = P^{(1)}\times\cdots\times P^{(n)}$ be defined as in (\ref{H def}). Then, similarly we have
\[ \ff_{S_n}(H) \simeq \frac{H}{\ker(\ff_{S_n})\cap H}\simeq \prod_{i\in\mathbb{I}} P^{(i)},\]
and so it has $p$-rank equal to $m$, which agrees with the $m$ defined after (\ref{H def}). For any $\sigma\in H$,  by projecting it onto the $i$th components for $i\in\mathbb{I}$, we obtain an element
\[ \sigma'\in\prod_{i\in\mathbb{I}} P^{(i)}\mbox{ such that }\theta_{\sigma} = \theta_{\sigma'},\mbox{ where }\theta_\sigma = \ff_{S_n}(\sigma)\mbox{ and }\theta_{\sigma'} = \ff_{S_n}(\sigma')\]
are as in (\ref{fsigma}). This element $\sigma'$ lies in $H$ and commutes with every \mbox{element of}

\par\noindent$\ker(\ff_{S_n})$. Hence, for all $i\in\bN_n\setminus X_0$, we may pick the $\sigma_i\in H$ defined prior to Lemma~\ref{g bound lem} such that it commutes with every element of $\ker(\ff_{S_n})$. Note that $\ff(\ker(\ff_{S_n}))\subset\Inn(G)$ by Proposition~\ref{out prop1} and Remark~\ref{remark'}. From Lemma~\ref{g bound lem}, we then deduce that (\ref{inequality}) holds.

\vspace{1.5mm}

Now, suppose that $\mathcal{N}$ is regular, which implies that $\fg$ is bijective by (\ref{regularity}). Note that $\proj_{\Aut}(\mathcal{N}) = \ff(G)$ and suppose further for contradiction that we have $\ff(G)\not\subset\Inn(G)$. This means that $m\geq1$ because $\ff(\ker(\ff_{S_n}))\subset\Inn(G)$. But then from Proposition~\ref{out prop2}, we deduce that $|T|$ can only be divisible by the primes $p\leq 3$, which is impossible by Burnside's theorem. It now follows that necessarily $\ff(G)\subset \Inn(G)$, and hence $\E'_\out(G,G)$ must be empty, which proves the second claim in Theorem~\ref{main thm}.

%
%

\section{Acknowledgments} 

The author would like to thank the referee for pointing out some small errors and unclear arguments in the original manuscript, which helped improve the exposition significantly.

\end{document}